\documentclass[12pt]{article}
\usepackage{amssymb}
\usepackage{amsmath}
\usepackage{latexsym}
\usepackage[T1,T2A]{fontenc}
\usepackage[cp1251]{inputenc}
\usepackage[english,ukrainian,russian]{babel}

\newtheorem{thrm}{Theorem}[section]
\newtheorem{corl}{Corollary}[section]
\newtheorem{lemm}{Lemma}[section]
\newtheorem{prop}{Proposition}[section]
\newtheorem{defn}{Definition}[section]
\newtheorem{exmp}{Example}[section]
\newtheorem{rmrk}{Remark}[section]
\newenvironment{proof}{{\bf Proof.}}{$\blacksquare$\medskip}

\newcommand{\Inf}{\mathrm{Inf}\,}
\begin{document}

\begin{center}
{\bf\Large A decomposition of directed graphs and the Tur{\' a}n problem}
\end {center}

\begin{center}
{\bf B.~V.~Novikov, L.~Yu.~Polyakova, G.~N.~Zholtkevich} (Karazin National University, Kharkov, Ukraine)
\end{center}

\medskip

\begin{abstract}
We consider vertex decompositions of (di)graphs which appear in
Auto\-mata Theory, and establish some their properties. Then we apply them to the
problem of forbidden subgraphs.
\end{abstract}

\medskip

\section*{Introduction}

This note has arisen from attempts to extend on pre-automata \cite{dnz} the con\-cepts
of regions and intervals used in the translation theory \cite{bi::dragon}. Unlike the
known model, the uniqueness  of the header of an interval is an unacceptable condition
for pre-automata. So we had to consider a generalized problem; and it was convenient to
collect obtained graph-theoretic results in a separate article.

General definitions and results are given in Section~\ref{sec1}. Regions and inter\-vals
are considered in Section~\ref{sec2} as a special case. Next we study the decompo\-sitions
of undirected graphs (Section~\ref{sec3}) and in particular consider their connection with maximal matchings. In Section~\ref{sec4} we study the main applica\-tion of decompositions --- the problem of forbidden graphs. Note that this problem can be posed also for digraphs. We hope
that in this case our construction will be even more useful.

\medskip

Mainly, we will use the definitions and notations of \cite{bol}.
Thus by the {\bfseries directed graph} (or {\bfseries digraph}) we
mean a pair $G = (V, E)$ where $V = V (G)$ is a set whose elements
are called {\bfseries vertices}; $E = E(G)\subset V \times V$ is a
binary relation whose elements are called {\bfseries arcs}.

For every vertex $v\in V$ define the sets ${\rm D^-}(v) =\{u\in
V\mid (u,v)\in E\}$ and ${\rm D^+}(v) =\{u\in V\mid (v,u)\in E\}$
whose elements are called the {\bfseries inputs} and {\bfseries
outputs} of the vertex  $v$ respectively. Note that loops are
included both in ${\rm D^-}(v)$ and ${\rm D^+}(v)$.

The numbers of inputs and outputs are denoted by ${\rm d^-}(v)$
and ${\rm d^+}(v)$ respectively.

The subgraph of (di)graph $G$ generated by a subset of vertices $U\subset V$ is denoted
by $G[U]$. A subset $U\subset V$ is called {\bfseries connected} if the graph $G[U]$ is
connected (i.\,e. for any two vertices  $u,v \in U$ there exists a directed path in
$G[U]$ starting at $u$ and ending at  $v$).

The symbol $\bigsqcup$ is used for the union of disjoint sets.

\section{Inflation and stability}\label{sec1}

Let $G = (V, E)$ be a digraph.
\begin{defn}\label{df::inflation}
An {\bfseries inflation} of a set $U\subset V$ is a set
$$
\Inf U = U\cup \{v\in V\mid \varnothing\ne {\rm D^-}(v)\subseteq U\}.
$$
\end{defn}

We need a property of the inflation:

\begin{prop}\label{prop::infmeet}
For any subsets  $X,Y\subset V(G)$
$$
\Inf X \cap \Inf Y= (X\cap\Inf Y)\cup (\Inf X \cap Y)\cup \Inf (X\cap Y).
$$
\end{prop}
\begin{proof} The inclusion $\supseteq$ is obvious. Prove the converse.
The case $v\in X\cup Y$ is also evident. Let $v\in (\Inf X \cap \Inf
Y)\setminus (X\cup Y)$. Then by the definition of inflation
$\varnothing \ne {\rm D^-}(v)\subset X \cap Y$, i.\,e. $v\in \Inf
(X\cap Y)$.
\end{proof}

We can consider the operator
$\Inf:\mathbf{2}^V\rightarrow\mathbf{2}^V: U\mapsto\Inf{U}$ and its
iterations $\Inf^n$ $(n>0)$. Suppose, in addition,  $\Inf^0 U=U$.

\begin{defn}\label{df::hyper}
A {\bfseries hyperinflation} of a subset $U$ is a set
$$
\Inf^{\infty} U = \bigcup_{n\ge 0} \Inf^n U.
$$
\end{defn}
Sometimes we say that $U$ is a hyperinflation if $U=\Inf^{\infty}
U'$ for some $U'$.

The following statement will be used bellow:

\begin{lemm}\label{lemm::entry} Let $U\subset V$ and $v\in V\setminus \Inf^{\infty} U$.
Then ${\rm D^+}(v)\cap\Inf^{\infty} U \subset U$.
\end{lemm}
\begin{proof}
The case ${\rm D^+}(v)\cap\Inf^{\infty} U =\varnothing$ is obvious.
Let $u\in {\rm D^+}(v)\cap\Inf^{\infty} U \ne\varnothing$, $u\not\in
U$. Then $u\in \Inf^n U = \Inf (\Inf^{n-1} U)$ for some $n\ge 0$.
This contradicts the fact that ${\rm D^-}(u)\ni v \not\in \Inf^{n-1}
U$.
\end{proof}

We define the notion of a hull which is close to the hyperinflation.

\begin{defn}\label{df::stable}
A set $U\subset V$ is called {\bfseries stable} if $\Inf U = U$.
\end{defn}

\begin{lemm}\label{lemm::intersec} An intersection of stable sets is stable.
\end{lemm}
\begin{proof} Let $U_i\ (i\in I)$ be stable sets,
$\displaystyle X=\bigcap_{i\in I}U_i$, and $v\in \Inf X\setminus
X$. By the definition of the inflation $\varnothing\ne {\rm
D^-}(v)\subseteq X\subseteq U_i$ for every $i\in I$. Therefore,
$v\in \Inf U_i=U_i$ whence $v\in X$.
\end{proof}

Since the set $V$ of vertices is stable, we have

\begin{corl}\label{corl::stable} For every vertex set $U\subseteq V$
there exists the smallest stable set ${\rm Hull}\,U$ containing
$U$.
\end{corl}

\begin{defn}\label{df::hull}
We say that ${\rm Hull}\,U$ is the {\bfseries hull} of a set $U$.
\end{defn}

It is clear that $ {\rm Hull}\,U $ is the intersection of all
stable sets containing $U$.

Consider connections between the introduced concepts. From
$U\subset\Inf U$ it follows $\Inf^{\infty}U \subset
{\rm Hull}\,U $. The reverse inclusion is true if and only if the
hyperinflation is stable. The following example shows that,
generally speaking, for infinite digraphs this does not hold.

\begin{exmp}\label{exmp1}
\rm Consider a digraph  $G=(V,E)$ such that
$$
V=\{0,1,2,\ldots\},\ \ \  E=\{(n,0)\mid n\ge 1\}\cup\{(n,n+1)\mid n\ge 1\}.
$$
Then $\Inf^n \{1\}=\{1,2,\ldots, n+1\}$, whence $\Inf^{\infty}
\{1\}=V\setminus \{0\}$. On the other hand, ${\rm
Hull}\,\{1\}=\Inf (\Inf^{\infty} \{1\})=V$.
\end{exmp}

\begin{defn}\label{df::locfin}
We call a digraph $G=(V,E)$ {\bfseries locally ${\rm d^-}$-finite}
if ${\rm d^-}(v) <\infty$ for all $v\in V$.
\end{defn}

\begin{prop}\label{prop::locfin}
If a digraph $G =(V, E)$ is locally ${\rm d^-}$-finite, then
$\Inf^{\infty} U = {\rm Hull}\,U$ for every $U \subset V$.
\end{prop}
\begin{proof} Suppose the contrary. Let $v \in \Inf (\Inf^{\infty} U)
\setminus \Inf^{\infty} U$. As $d^-(v)<\infty$ and $\varnothing \ne
{\rm D^-}(v) \subseteq \Inf^{\infty} U$, we have  ${\rm D^-} (v)
\subset \Inf^n U$ for some $n \ge 0$. But then $v \in \Inf^{n +1} U
\subset \Inf^{\infty} U$ contrary to assumption.
\end{proof}

The statement similar to Lemma~\ref{lemm::entry} is not true for
the hull:

\begin{exmp}\label{exmp2}
\rm Add the vertex $-1$ and the arc $(-1,0)$ to the digraph $G$ from
Example~\ref{exmp1}. Obviously $\{0\}={\rm D^-}(-1)\cap{\rm
Hull}\,\{1\}\ne \{1\}$.
\end{exmp}

Now we introduce the main definition of this article:

\begin{defn}\label{df::decomp}
A {\bfseries decomposition} of a digraph $ G = (V, E) $ is a set of
its subgraphs $G_i = (V_i, E_i)$ $(i\in I)$ such that

(i) $V_i$ are hyperinflations,

(ii) $V$ is a disjoint union of  $V_i$,

(iii) $E_i$ is a restriction of $E$ to $V_i$.
\end{defn}

\begin{rmrk}\label{rmrk::decomp}
\rm Our definition differs from the one given, for example, in
\cite{west}, where a decomposition means a partition of $E(G)$.
\end{rmrk}

{\it Till the end of Section~\ref{sec2} we assume that some locally $
{\rm d^-} $-finite  graph $G(V, E)$ is fixed.}

\begin{thrm}\label{thrm::hypmeet} For any subsets $X,Y\subset V$
$$
\Inf^\infty X\cap \Inf^\infty Y = \Inf^\infty [(X\cap \Inf^\infty Y) \cup (\Inf^\infty
X\cap Y)].
$$
\end{thrm}
\begin{proof} The inclusion $\supseteq$ is clear. Indeed,
$$
(X\cap \Inf^\infty Y) \cup (\Inf^\infty X\cap Y) \subset
\Inf^\infty X\cap \Inf^\infty Y;
$$
and the set $\Inf^\infty X\cap \Inf^\infty Y$ is stable because of
the locally ${\rm d^-}$-finiteness of the original graph and
Lemma~\ref{lemm::intersec}.

Let $x\in \Inf^\infty X \cap \Inf^\infty Y$. Then there are
integers $ m, n \ge 0 $ such that $x\in \Inf^m X \cap \Inf^n Y$,
$x\not\in \Inf^{m-1} X \cup \Inf^{n-1} Y$, and
$$
{\rm D^-}(x) \subset \Inf^{m-1} X \cap \Inf^{n-1} Y \subset \Inf^\infty X \cap
\Inf^\infty Y.
$$
Use the induction on $m+n$ to show that
\begin{equation}\label{eq1}
x\in \Inf^\infty [(X\cap \Inf^\infty Y) \cup (\Inf^\infty X\cap Y)].
\end{equation}

If $m=0$ or $n=0$, then $x\in X\cup Y$; hence \eqref{eq1} is hold.

If $m=n=1$, then ${\rm D^-}(x) \subset X \cap Y$. Therefore, $x\in
\Inf (X\cap Y) \subset \Inf^\infty (X \cap Y) \subset \Inf^\infty
[(X\cap \Inf^\infty Y) \cup (\Inf^\infty X\cap Y)]$.

Consider the general case. Since ${\rm D^-}(x) \subset \Inf^{m-1}
X \cap \Inf^{n-1} Y \subset \Inf^\infty X \cap \Inf^\infty Y$, by
the induction assumption
$$
{\rm D^-}(x) \subset \Inf^\infty [(X\cap \Inf^\infty Y) \cup (\Inf^\infty X\cap Y)].
$$
Consequently \eqref{eq1} is true.
\end{proof}

\begin{corl}\label{corl::hypmeet} $\Inf^\infty X\cap \Inf^\infty Y =\varnothing$
if and only if
$$
\Inf^\infty X\cap Y = X\cap \Inf^\infty Y =\varnothing.
$$
\end{corl}

As we will see below, the hyperinflations of connected subsets are
of particular interest.

\begin{thrm}\label{thrm::conn}
Let $X,Y\subset U$, $\Inf^\infty X\cap \Inf^\infty Y
\ne\varnothing$, and $X$ is connected. If $\Inf^\infty X\cap Y
=\varnothing$, then $\Inf^\infty X\subset \Inf^\infty Y$.
\end{thrm}
\begin{proof}
Corollary~\ref{corl::hypmeet} implies  $X\cap \Inf^\infty
Y\ne\varnothing$. If $X$ is a singleton, then  the statement is
obvious. Let $|X|>1$. We choose the smallest $n$ such that $X\cap
\Inf^n Y\ne\varnothing$ (it follows from the conditions that $n\ge
1$). Let $x\in X\cap \Inf^n Y$. Since $x\not\in Y$, we have ${\rm
D}^-(x) \subset \Inf^{n-1} Y$. But ${\rm D}^-(x) \cap
X\ne\varnothing$, because $X$ is connected. Therefore, $X\cap
\Inf^{n-1} Y\ne\varnothing$; that contradicts the choice of $n$.
\end{proof}

Theorem~\ref{thrm::conn} allows us to construct (not uniquely)
decompositions of finite graphs whose components are hyperinflations
of connected sets. Describe the process of constructing in detail.

Let $ V $ be a set of vertices of the graph.  We take as $V_1$ an
arbitrary connected subset (for example, a vertex). Suppose we have
taken components $V_1, \dots, V_k$ with disjoint hyperinflations. In
the complement
$$
\displaystyle U= V\setminus \bigsqcup_{i\le k}
\Inf^\infty V_i
$$
choose an arbitrary connected subset $W$. If
$\Inf^\infty W\cap \Inf^\infty V_j \ne\varnothing$ for some $1\le
j\le k$, then $\Inf^\infty V_j \subset \Inf^\infty W$ by
Theorem~\ref{thrm::conn}. In this case replace all  $V_j$ by $ W $
and go on to the choice of the next component. If $\displaystyle
\Inf^\infty W\cap \bigsqcup_{i\le k} \Inf^\infty V_i =\varnothing$,
then we put $V_{k+1}=W$ and continue the process.

\section{Regions and intervals}\label{sec2}

Using the terminology of Computer Science \cite{bi::dragon} we
introduce the following

\begin{defn}\label{df::reg-int}
A subset $ U \subset V $ is said to be a {\bfseries region} if $U =
\Inf^{\infty} \{x \}$ for some $ x\in V $. In this case $x$ is
called a {\bfseries heading} of $ U $. A region is called an
{\bfseries interval} if it is not contained in any other region.
\end{defn}

Generally speaking, the region can have multiple headings. A
sufficient condition for the uniqueness of the heading (this demand
is essential for Computer Science) is obtained directly from
Lemma~\ref{lemm::entry}:

\begin{prop}\label{prop::entry}
Let $U\subset V$ be a region with a heading $x$. If there exists
$y\in V\setminus U$ such that ${\rm D^+}(y)\cap U \ne\varnothing$,
then $x$ is uniquely defined, i.\,e.  ${\rm D^+}(y)\cap U =\{x\}$.
\end{prop}

Since a singleton is connected, it follows directly  from
Theorem~\ref{thrm::conn}:

\begin{prop}\label{prop::empty}
If two regions have a nonempty intersection, then one of them
contains the other.
\end{prop}

Now we can state the main result about intervals of finite
digraphs:

\begin{thrm}\label{thrm::decomp}
Every digraph $G = (U, E)$ with the finite set of vertices has the
unique decomposition whose components are intervals.
\end{thrm}

\begin{proof} The existence of such a decomposition follows directly from
Pro\-position~\ref{prop::empty}; the uniqueness follows from the
maximality of each interval.
\end{proof}

Consider two extreme cases.

\begin{prop}\label{prop::decomp-one}
All components of an interval decomposition of a digraph are
singletons if and only if ${\rm d^-}(v) = 1$ implies $ {\rm D^-}
(v) = \{(v, v) \} $ for any vertex $v$.
\end{prop}
\begin{proof} Let $G=(V,E)$ be a considered digraph. It is clear that all
its components are singletons if and only if $|\Inf \{v\}|=1$ for
all  $v\in V$. If ${\rm d^-}(v)= 1$ and ${\rm D^-}(v) \ne\{(v,v)\}$,
then there is a vertex $u\ne v$ such that $(u,v)\in E$ and $v\in
\Inf \{u\}$. This implies $|\Inf \{u\}|>1$.

Conversely, suppose that the restriction on ${\rm D^-}$ from the
proposition con\-ditions is hold and $|\Inf \{v\}|>1$ for some $v$. If
$u\in \Inf \{v\}\setminus \{v\}$, then by definition of inflation
${\rm d^-}(v)= 1$; in addition, the arc from ${\rm D^-}(v)$ can not
be a loop.
\end{proof}

Now assume that $G =(V, E) $ is finite and its decomposition
consists of only one component, i.\,e. digraph is an interval. Let
$x$ be a heading of this interval (in general, not the only one),
i.\,e. $G = \Inf^\infty \{x \} = \Inf^n \{x \} $ for some $n> 0$.

\begin{defn}\label{df::jet}
A finite digraph $\mathbf{H}=(W,F)$ with a partition $\displaystyle
W=\bigsqcup_{i=1}^nW_i$ $(n\in \mathbb{N})$ is called  a {\bfseries
jet} if it satisfies the following conditions:
\begin{quotation}
(i) if $i\le j$, then $(W_j\times W_i)\cap F=\varnothing$;

(ii) for each $j\ge 2$ and every vertex $x\in W_j$ there exist
$y_i\in W_i$ $(1\le i<j)$ forming a directed path
$$
y_1\to y_2\to\ldots \to y_{j-1}\to x.
$$
\end{quotation}
\end{defn}

\begin{prop}\label{prop::decomp-all}
Let $\mathbf{H}=(W,F)$ be a jet, $x$ be an element not contained
in $W$, and $V=W\cup \{x\}$. Choose an arbitrary subset
$$
C\subset \bigsqcup_{i>1}\left(W_i\times \{x\}\right)\cup \{(x,x)\}
$$
and put $E=F\cup C\cup (\{x\}\times W_1)$. Then in the digraph
 $G=(V,E)$ the subset $V$ is an interval with a heading $x$.
\end{prop}

\begin{proof}
Denote $W_0=\{x\}$, $\displaystyle Z_j=\bigsqcup_{i=0}^j W_i$ and
verify that $Z_j=\Inf Z_{j-1}$. The inclusion $\supseteq$ is
evident. Conversely, suppose that $y\in W_j$ $(j>0)$. By condition
(ii) of Definition~\ref{df::jet} ${\rm D^-}(y)\ne\varnothing$. By
condition (i) $z\in {\rm D^-}(y)$ implies $z\in W_k$ for some
$k<j$. It means that ${\rm D^-}(y)\subset Z_{j-1}$.

It is easy to see that $\Inf \{x\}=W_1$; and thus
$V=\Inf^n\{x\}=\Inf^\infty\{x\}$.
\end{proof}

The converse is true. Moreover:

\begin{prop}\label{prop::decomp-gen}
Let $\{G_j=(V_j,E_j)\mid 1\le j\le N\}$ be an interval
decomposi\-tion of a finite digraph $G=(V,E)$ and $V_j=\Inf^\infty
\{x_j\}$. Then every subgraph $(V_j\setminus\{x_j\},
E_j|_{V_j\setminus\{x_j\}})$ with the partition
$$
V_j\setminus\{x_j\}=\bigsqcup_{i_j=1}^\infty
(\Inf^{i_j}\{x_j\}\setminus\Inf^{i_j-1}\{x_j\})
$$
is a jet.
\end{prop}
\begin{proof}
Consider an interval $V_k$ and put
$W_i=\Inf^i\{x_k\}\setminus\Inf^{i-1}\{x_k\}$. If $(u,v)\in
(W_j\times W_i)\cap E$ for $1\le i\le j$ then $v\in {\rm D^+}(u)
\not\subset\Inf^{i-1}\{x_k\}$. Hence for $V_j\setminus\{x_j\}$
condition (i) of Definition~\ref{df::jet} is hold. Condition (ii)
is obvious.
\end{proof}

\section{Undirected graphs}\label{sec3}

In this section we assume that $G = (V, E) $ is a finite
undirected connected graph without loops. We will use the
notations ${\rm D}(v)$ and ${\rm d}(v)$ instead of ${\rm
D}^\pm(v)$ and ${\rm d^\pm}(v)$ respectively.

In the undirected case the description of a hyperinflation is
simplified:

\begin{prop}\label{prop::nondir-inf}
$\Inf^\infty U=\Inf U$ for every subset $U\subset V$.
\end{prop}
\begin{proof}
Let $x\in \Inf^\infty U\setminus U$. Then $x\in \Inf^n U$ for some
$n\ge 1$ and $(y,x)\in E$ for some $y\in \Inf^{n-1} U$. But this
is impossible for $n>1$, otherwise, $x\in {\rm D^+}(y)={\rm
D^-}(y)\subset\Inf^{n-2} U$. Therefore, $n=1$ and $x\in \Inf U$.
\end{proof}

Thus in what follows we may talk about the inflation rather than
the hyperinflation and use the appropriate notations.

Consider some variants of decompositions. We will write them in
the form of
\begin{equation}\label{eq2}
V=\left(\bigsqcup_i \Inf V_i\right)\sqcup U
\end{equation}
where $G[V_i]$ are graphs of some (fixed) class and $U$ is a
subset of singleton components.

First, in the process described after Theorem~\ref{thrm::conn}  we can choose
non-singleton connected subsets as $V_j$ until this is possible. Let components
$V_1,\dots, V_k$ be chosen in such way and in $\displaystyle U= V\setminus
\bigsqcup_{i\le k} \Inf V_i$ there are no any connected components other than vertices.
It means that $U$ is completely disconnected. Moreover ${\rm D}(v)\subset V_i$ for every
$v\in \Inf V_i \setminus V_i$. This proves

\begin{prop}\label{prop::conn-decomp}
Each connected graph $G =(V, E)$ with a finite set of vertices has
the decomposition of form~\eqref{eq2} where $V_i$ are non-singleton
connec\-ted subsets and $U$ is completely disconnected (possibly
empty). Moreover $\displaystyle \left(\bigsqcup_i \Inf V_i \setminus
V_i\right)\cup U$ is completely disconnected subset.
\end{prop}

Another type of a decomposition is obtained if we choose two-element
connected subsets, i.\,e. arcs, as  $ V_1, \dots, V_k$. Clearly, the
proof will not change, and we get

\begin{corl}\label{corl::match} Each connected graph $G = (V, E)$
with a finite set of vertices has the decomposition of
form~\eqref{eq2} where $V_i$ are arcs\footnote{We identify here an
arc and the connected set of its vertices.}, $U$ is completely
discon\-nected (possibly empty) subset as well as $\displaystyle
\left(\bigsqcup_i \Inf V_i \setminus V_i\right)\cup U$.
\end{corl}

Recall that a {\bf matching} of a graph is a set of pairwise
non-adjacent edges, i.\,e. the arcs that have no common vertices. A
matching is said to be {\bf maximal}, if it is not contained in any
other matching of the graph, and is said to be the {\bf greatest},
if it contains the maximum number of arcs.

Decompositions of Corollary~\ref{corl::match} are characterized in
terms of matchings:

\begin{thrm}\label{thrm::match}
For a finite connected graph $G=(V,E)$ with the decomposi\-tion of form~\eqref{eq2}
satisfying the conditions of Corollary~\ref{corl::match} the arcs $V_1,V_2,\dots$ form
the maximal matching. Conversely, if $\{V_1,V_2,\dots\}$ is a maximal mat\-ching, then
expression~\eqref{eq2} is a decomposition satisfying the conditions of
Corollary~\ref{corl::match}.
\end{thrm}
\begin{proof}
By construction different $V_i$ and $V_j$ have no common vertices, there\-fore
$\{V_1,V_2,\dots\}$ is a matching. The complete disconnectedness of
$$
\displaystyle \left(\bigsqcup_i \Inf V_i \setminus V_i\right)\cup U
$$
implies its maximality.

Conversely, let $\{V_1,V_2,\dots\}$ be a maximal matching and
$V_i=(x_i,y_i)$. Suppose that $\Inf V_i\cap \Inf
V_j\ne\varnothing$ ($i\ne j$). According to
Corollary~\ref{corl::hypmeet} we can assume that $\Inf V_i\cap
V_j\ne\varnothing$. Since $V_i\cap V_j=\varnothing$, either $x_j$
or $y_j$ is contained in $\Inf V_i\setminus V_i$.

If, for example,  $x_j\in\Inf V_i\setminus V_i$, then $y_j\in {\rm
D^-}(x_j)\subset V_i$; that is impossible. Similarly $y_j\in \Inf
V_i\setminus V_i$ implies $x_j\in V_i$. Hence $\Inf V_i\cap \Inf
V_j=\varnothing$ for all $i\ne j$.
\end{proof}

We deal with another variant of a decomposition in the next
section.

\section{Forbidden subgraphs}\label{sec4}

In this section we apply a decomposition to the well-known forbidden
sub\-graphs problem. This direction began with Tur\'an's work
\cite{tur} about the number of edges in the graph that does not
contain any clique of given order. A good overview is given in
\cite{bol}. Among the recent articles we mention also \cite{ehss}.

In general, the problem statement is as follows:

Let $H$ be a fixed finite graph ({\bf forbidden graph}). Find the
least upper bound ${\rm ex} (p, H)$ for the number of arcs of finite
graphs with $p$ vertices, not containing $H$ as a subgraph (such
graphs are called {\bf $H$-free}).

We use the number-theoretic functions ``floor'' $\lfloor
x\rfloor$, ``ceiling'' $\lceil x\rceil$, and fractional part
$\{x\}$. Recall that
$$\lfloor x\rfloor=x-\{x\}; \qquad \lceil x\rceil=x+\{-x\}.$$

For $K_3$ (the complete graph of order 3) ${\rm ex}(p,K_3)=
\lfloor\frac{p^2}{4}\rfloor$ \cite[Theorem I.2]{bol}. We obtain a similar
evaluation for the graph $H$ of the form

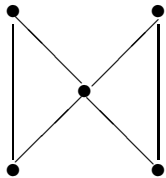
\begin{picture}(200,80)(-40,0)
\put(111,60){$\bullet$} \put(166,60){$\bullet$} \put(138,30){$\bullet$}
\put(111,0){$\bullet$} \put(166,0){$\bullet$}

\put(140,36){$\line(-1,1){25}$} \put(144,35){$\line(1,1){25}$}
\put(115,6){$\line(1,1){25}$} \put(167,5.5){$\line(-1,1){25}$}
\put(114,7){$\line(0,1){51}$} \put(169,7){$\line(0,1){51}$}
\end{picture}

\noindent Hereinafter $H$ denotes just this graph.
Following~\cite{west} we call it a ``bowtie''.

A sequence of different vertices $U=$ $\{v_1,\ldots,v_n\}\subset V$
of $G(V,E)$ is said to be a {\bf path} if $(v_i,v_{i+1})\in E$ for
all $1\le i<n$.

If $U_1, U_2$ are two disjoint subsets of vertices, then ${\rm
d}(U_1, U_2)$ denotes the number of arcs connecting vertices of $U_1
$ with those of $U_2$.

The {\bf volume} of $G=(V,E)$ is a pair ${\rm vol}\, G=(p,q)$
where $p$ is a number of vertices and $q$ is a number of arcs in
$G$.

Our main result in this section is the following:
\begin{thrm}\label{thrm::H}
${\rm ex}(p,H)= \lfloor\frac{p^2}{4}\rfloor+1$ for $p> 4$.
\end{thrm}

To build an $H$-free graph with exactly
$\lfloor\frac{p^2}{4}\rfloor+1$  arcs, it is sufficient to consider
$K_{\lfloor\frac{p}{2}\rfloor, \lceil\frac{p}{2}\rceil}$ (the
complete bipartite graph with partite sets containing
$\lfloor\frac{p}{2}\rfloor$ and $\lceil\frac{p}{2}\rceil$ vertices )
and to draw one more arc in one of its partite sets.

The example of the graph $K_4$ shows that for $p = 4$ the
statement of Theorem~\ref{thrm::H} is violated.

The remaining part of the article will be devoted to the proof of
the proposition which implies, taking into account the facts
mentioned above, the theorem.

\begin{prop}\label{prop::pereformulirovka}
Let  $G$ be an $H$-free graph which is not isomorphic to $K_4$ and ${\rm
vol}\,G=(p,q)$. Then $q\leq \frac{p^2}{4}+1$.
\end{prop}

First, make sure that it is enough to prove this statement for
connected graphs.

\begin{lemm}\label{lemm::connect}
Let $G$ be a disconnected graph,  ${\rm vol}\,G=(p,q)$ and $G_j$
$(j=1,\dots,n$, $n\geq 2)$ be all of its connected components with
the volumes ${\rm vol}\,G_j=(p_j,q_j)$. If $q_j\leq
\frac{p_j^2}{4}+1$ for all $j$, then $q\leq \frac{p^2}{4}+1$.
\end{lemm}
\begin{proof} It is clear that we can consider only the case $n=2$.
If $p_1=p_2=1$ then $q_1=q_2=0$, and the lemma is true. Otherwise, $p_1p_2\geq 2$ implies
$\frac{(p_1+p_2)^2}{4}+1\geq
\left(\frac{p_1^2}{4}+1\right)+\left(\frac{p_2^2}{4}+1\right)$.
\end{proof}

In what follows  we  assume that $G$ is a finite and connected
graph and do not indicate that specially.

Prove some auxiliary statements.

\begin{lemm}\label{lemm::path2}
Let $G=(V,E)$ be an $H$-free graph, $U=\{v_1,\ldots,v_l\}$ be a
path in $G$ and $x\in V\setminus U$. Then ${\rm d}(x,U)\leq
\lceil\frac{l}{2}\rceil+1$. If the equality holds, then

1) there exist vertices $v_j, v_{j+1}$ adjacent to $x$;

2) either $v_1$ or  $v_l$ is adjacent to $x$, if $l$ is even;

3) both $v_1$ and $v_l$ are adjacent to $x$, if $l$ is odd.
\end{lemm}
\begin{proof}
Let $l$ be even, $l=2m$.  Suppose that ${\rm d}(x,U)\ge
\lceil\frac{l}{2}\rceil +2=m+2$. By the pigeonhole principle among
$m$ pairs $(v_1,v_2), (v_3, v_4),\ldots,$ $(v_{2m-1},v_{2m})$ there
exist $(v_i,v_{i+1}),(v_j,v_{j+1})$ such that four arcs outgoing
from $x$ end in them. At the same time $i\ne j+1, j\ne i+1$, hence
$x,v_i,v_{i+1},v_j,v_{j+1}$ form a subgraph isomorphic to $H$.

If there are no $v_j, v_{j+1}$ adjacent to $x$, then ${\rm
d}(x,U)\leq m<\lceil\frac{l}{2}\rceil +1$. This implies  1) for
even $l$.

Since for the path $U'=\{v_2, \dots, v_{l-1}\}$ consisting of $l-2$ vertices ${\rm d}(x,
U')\leq \lceil\frac{l-2}{2}\rceil+1=\lceil\frac{l}{2}\rceil$, it follows from ${\rm
d}(x,U)=\lceil\frac{l}{2}\rceil+1$ that either $v_1$ or $v_l$ is adjacent to $x$;
therefore 2) is hold.

Let $l=2m+1$. As it was proved above, for $U'=\{v_1,\ldots,
v_{2m}\}$ the inequality ${\rm d}(x,U')\leq m+1$ holds. Hence ${\rm
d}(x,U)\leq m+2=\lceil\frac{l}{2}\rceil+1$; and the equality is
possible only in the case when the vertex $u_{2m+1}$ is adjacent to
$x$ and there exist vertices $v_j, v_{j+1}$ adjacent to $x$. To
complete the proof of 3) it suffices to consider the path
$\{v_2,\ldots, v_{2m+1}\}$ and to deduce that $v_1$ is adjacent to
$x$.
\end{proof}

A path $\{v_1, \dots, v_{l}\}$ in the graph $G=(V,E)$ is called
\textbf{premaximal} if there exists a vertex $v_{l+1}\in V\setminus
U$ such that the path $\{v_1, \dots, v_{l}, v_{l+1}\}$ is maximal,
i.\,e. has the maximum possible length.

\begin{lemm}\label{lemm::Inflation}
Let $G=(V,E)$ be not completely disconnected and $U$ be a
premaximal path. Then ${\rm Inf}\, U\ne U$.
\end{lemm}
\begin{proof}
Let $U=\{v_1,\ldots,v_l\}$ and $U'=\{v_1,\ldots,v_{l+1}\}$ be a maximal path in  $G$.
Then ${\rm D}(v_{l+1})\subset U$. Therefore $v_{l+1}\in \Inf U\ne U$.
\end{proof}

\begin{lemm}\label{lemm::Almost max_path}
Let $U=\{v_1, \dots , v_l\}$ be a premaximal path in the $H$-free graph $G=(V,E)$. If
$x\in V\setminus U$ and ${\rm d}(x, U)=\lceil\frac{l}{2}\rceil+1$, then for every vertex
$y\in V\setminus U$ such that $y\ne x$, inequality ${\rm d}(y, U)\leq
\lceil\frac{l}{2}\rceil-1$ holds.
\end{lemm}
\begin{proof}
Assume the contrary, let ${\rm d}(y, U)\geq
\lceil\frac{l}{2}\rceil$. Note that by Lemma~\ref{lemm::path2} there
is a pair of vertices $v_j, v_{j+1}$ adjacent to $x$. Then $y$ can
not be adjacent to $v_1$, otherwise the path $\{y, v_1,\dots ,
v_j,x,v_{j+1}, \dots, v_l\}$ is longer than maximal. Similarly $y$
is not adjacent to $v_l$. Hence for the path $U'=\{v_2, \dots,
v_{l-1}\}$ the inequality ${\rm d}(y, U')\geq
\lceil\frac{l}{2}\rceil=\lceil\frac{l-2}{2}\rceil+1$ holds.
Therefore by Lemma~\ref{lemm::path2} we can find vertices  $v_i,
v_{i+1}$ adjacent to $y$. Without loss of generality, we can also
suppose, in view of statements 2), 3) of Lemma~\ref{lemm::path2},
that $x$ is adjacent to $v_1$. Then the path $\{x,v_1, \dots, v_i,
y, v_{i+1}, \dots v_l\}$ is longer than maximal path, a
contradiction.
\end{proof}

\begin{corl}\label{corl::almost_max_path}
Let $U=\{v_1, \dots , v_l\}$ be a premaximal path in the $H$-free
graph $G=(V,E)$ and $|{\rm Inf}\,U|=p$. If $p-l\geq 2$, then ${\rm
d}({\rm Inf}\,U\setminus U, U)\leq
(p-l)\left\lceil\frac{l}{2}\right\rceil.$
\end{corl}

Let $U=\{v_1, \dots , v_l\}$ be a path. We put $U^{(j)}=U\setminus
\{v_j\}$.

\begin{lemm}\label{lemm::change}
Let $l\geq 3$ and $U=\{v_1, \dots , v_l\}$ be a path in the $H$-
and $K_4$-free graph $G=(V,E)$. Let $x\in V\setminus U$ and ${\rm
d}(x, U)=\lceil\frac{l}{2}\rceil+1$. If the number of arcs of the
subgraph $G[U]$ does not exceed $\frac{l^2}{4}+1$, then there
exists a vertex $v_j$ for which ${\rm d}(v_j,
U^{(j)}\cup\{x\})\leq \lceil\frac{l}{2}\rceil$.
\end{lemm}
\begin{proof}  Assume the contrary: for every vertex $v_j$
the inequality
$$
{\rm d}(v_j, U^{(j)}\cup\{x\})\geq
\left\lceil\frac{l}{2}\right\rceil+1
$$
holds. By
hypothesis $G[U\cup\{x\}]$ contains no more than
$\frac{l^2}{4}+\lceil\frac{l}{2}\rceil+2$ arcs. The assumption implies:
$$
\frac{1}{2}(l+1)\left(\left\lceil\frac{l}{2}\right\rceil+1\right)
\le \frac{l^2}{4}+\left\lceil\frac{l}{2}\right\rceil+2.
$$
Hence for even $l$ we have $l\le 6$ and for odd $l$ we have $l\le 3$. Therefore it is
sufficient to consider the cases $l=3,4,6$.

Let $A$ be a set of vertices of  $U$ which are adjacent to $x$, and
$\overline{A}=U\setminus A$.

If  $l=3$, then ${\rm d}(x, U)=\lceil\frac{3}{2}\rceil+1=3$; hence
$A=U=\{v_1, v_2, v_3\}$. Then the assumption implies that the
graph $G$ is isomorphic to $K_4$; this contradicts the condition.

Let $l=4$. Without loss of generality, we can assume that
$\overline{A}=\{v_3\}$ or $\overline{A}=\{v_4\}$. Then $v_1, v_2\in
A$. Hence $v_2$ and $v_4$ are not adjacent, otherwise, $G$ contains
a ``bowtie''. By assumption ${\rm d}(v_4, U^{(4)}\cup\{x\})\geq 3$,
i.\,e. $v_4$ is adjacent to $v_1$ and $x$, and
$\overline{A}=\{v_3\}$. Then $(v_1, v_3)\not \in E$, otherwise the
vertices of $U\cup\{x\}$ form a ``bowtie''. Therefore, ${\rm
d}(v_3, U^{(3)}\cup\{x\})< 3$ contrary to assumption.

Let $l=6$. Then  $\mid \overline{A} \mid=2$. Without loss of
generality, by Lemma~\ref{lemm::path2} we can suppose that $x$ is
adjacent to $v_1$.

First, assume that  $v_6\not\in A$.  Since $G$ is $H$-free, it
follows that either $A=\{v_1, v_2, v_3, v_5\}$ or $A=\{v_1, v_3, v_4,
v_5\}$. According to the assumption  ${\rm d}(v_6, U^{(6)})= {\rm
d}(v_6, U^{(6)}\cup\{x\})=4$. Therefore, as well as for $x$, there
are two variants: the set of vertices adjacent to  $v_6$ equals
either $\{v_1, v_2, v_3, v_5\}$ or $\{v_1, v_3, v_4, v_5\}$.
Checking straightforwardly four cases, we get a contradiction. So
$v_6\in A$.

Note that the cases $\overline{A}=\{v_3, v_4\}$,
$\overline{A}=\{v_3, v_6\}$ and $\overline{A}=\{v_5, v_6\}$ are
impossible, since $G$ is $H$-free. We will obtain the contradiction
in every of the remaining variants:

Let $\overline{A}=\{v_2, v_3\}$. Note that $(v_2, v_4), (v_3,
v_5)\not\in E$, since $G$ is $H$-free. Then assumption implies that
$v_2$ and $v_3$ are adjacent to $v_6$, hence the vertices $\{v_2,
v_3, v_6, v_5, x\}$  form a ``bowtie''. The case
$\overline{A}=\{v_4, v_5\}$ is similar.

Let $\overline{A}=\{v_2, v_4\}$. Note that $(v_3, v_5), (v_1,
v_3)\not \in E$. The assumption implies that $v_3$  is adjacent to
$v_6$, hence, $(v_4, v_6), (v_4, v_2)\not \in E$ and ${\rm d}(v_4,
U^{(4)}\cup\{x\})<4$ contrary to the assumption. The case
$\overline{A}=\{v_3, v_5\}$ is similar.

Let $\overline{A}=\{v_2, v_5\}$. Then $(v_4, v_6), (v_1, v_6)\not
\in E$ and the assumption implies that $v_2$ and $v_3$ are adjacent
to $v_6$. Therefore, the vertices $\{v_2, v_3, v_6, v_4, x\}$ form a
``bowtie''.
\end{proof}

\begin{prop}\label{prop::diam_2}
Let  $G$ be a connected graph,  ${\rm vol}\,G=(p, q)$, and the
length of the maximal path in $G$ does not exceed 2. Then  $q\leq
\frac{p^2}{4}+1$.
\end{prop}
\begin{proof} Note that the graph $G$ satisfying the condition is
isomorphic to $K_p$ for $ p \leq $ 3 or to $K_{1, p-1}$. The
inequality can be proved by immediate check.
\end{proof}

Now we are ready to prove
Proposition~\ref{prop::pereformulirovka}.

\noindent \begin{proof} First, let $G=(V,E)$ be a $H$- and
$K_4$-free graph, ${\rm vol}\, G=(p,q)$, and the length of the
maximal path in $G$ is greater than 2. Construct a decomposition
choosing  the pathes without self-intersections as $V_i$ and
taking a premaximal path as the first component $V_1$.

Let ${\rm vol}\,G[ V_i]=(l_i,m_i)$, ${\rm vol}\,G[ \Inf
V_i]=(p_i,q_i)$ ($i=1, \dots, n$). Then  $p_1>l_1$ by
Lemma~\ref{lemm::Inflation} and $l_1\geq 3$ by assumption.

Use an induction on $p$.

We consider separately the case $n = 1$.  We omit the indices in the
notations, thus, $p=p_1$, $q=q_1$, $l=l_1$, $m=m_1$.

Let $p>l+1$.  By the induction assumption $m\leq
\lfloor\frac{l^2}{4}\rfloor+1$, and by
Corollary~\ref{corl::almost_max_path}
$q-m\leq(p-l)\lceil\frac{l}{2}\rceil$. Then
\begin{eqnarray} \label{eq::Thm_H_3}
\frac{p^2}{4}+1 -q & = & \frac{p^2}{4}+1 -m - (q-m)\nonumber \\
& \ge & \frac{p^2}{4}+1 - \left\lfloor\frac{l^2}{4}\right\rfloor
-1-(p-l)\left\lceil\frac{l}{2}\right\rceil  \\
& = & \frac{p^2}{4}
- \frac{l^2}{4} + \left\{\frac{l^2}{4}\right\} -(p-l)\frac{l}{2}
-(p-l)\left\{-\frac{l}{2}\right\} \nonumber\\
& = & \frac{(p-l)^2}{4} + \left\{\frac{l^2}{4}\right\}
-(p-l)\left\{\frac{l}{2}\right\}=\left(\frac{p-l}{2}-
\left\{\frac{l}{2}\right\}\right)^2\geq 0. \nonumber
\end{eqnarray}

Let $p=l+1$. Note that there is a vertex $x\in {\rm Inf}\, V_1$ such that ${\rm d}(x,
{\rm Inf}\, V_1\setminus\{x\})\leq \lceil\frac{l}{2}\rceil$. Indeed,  $q-m\leq
\lceil\frac{l}{2}\rceil+1$ by Lemma~\ref{lemm::path2}. Therefore the only vertex of the
set ${\rm Inf}\, V_1\setminus V_1$ can be taken as $x$ or, by Lemma~\ref{lemm::change},
$x$ can be chosen in $V_1$.

Let ${\rm vol}\,G[{\rm Inf}\, V_1\setminus\{x\}]=(l,s)$. Then
$q-s\leq \lceil\frac{l}{2}\rceil$. Again by the induction
assumption $s\leq\lfloor\frac{l^2}{4}\rfloor+1$. So we have
\begin{eqnarray} \label{eq::Thm_H_4}
\frac{p^2}{4}+1 -q & = & \frac{(l+1)^2}{4}+1 -s - (q-s) \nonumber \\
& \ge & \frac{(l+1)^2}{4}+1 -
\left\lfloor\frac{l^2}{4}\right\rfloor-1
-\left\lceil\frac{l}{2}\right\rceil.
\end{eqnarray}
Note that expression~\eqref{eq::Thm_H_4} is obtained
from~\eqref{eq::Thm_H_3} by substitution $p=l+1$, hence, it is
nonnegative.

Thus for $n = 1$ the statement is proved.

Let $n\geq 2$. Put $$p'=p-p_1=\sum_{i=2}^n p_i,\quad
q'=q-q_1=\sum_{i=2}^n q_i+\sum_{1\leq i< j\leq n}{\rm d}(V_i,
V_j).$$

Applying the induction assumption to the subgraph
$G[V_1\sqcup\displaystyle\bigsqcup_{j=2}^n {\rm Inf}\,V_j]$, we have $m_1+q'\leq
\left\lfloor\frac{(l_1+p')^2}{4} \right\rfloor +1$. Moreover
$q_1-m_1\leq(p_1-l_1)\left(\lceil\frac{l_1}{2}\rceil+1\right)$ by
Lemma~\ref{lemm::path2}. Then
\begin{eqnarray} \label{eq::Thm_H_5}
\frac{p^2}{4}+1 -q & =   & \frac{(p_1+p')^2}{4}+1 -(m_1+q') - (q_1-m_1)  \nonumber \\
                   &\geq & \frac{(p_1+p')^2}{4}+1 - \frac{(l_1+p')^2}{4}+
                   \left\{\frac{(l_1+p')^2}{4}\right\} -1 -(q_1-m_1)\nonumber\\
                   &\geq & \frac{p_1^2}{4}+\frac{(p_1-l_1)p'}{2}  -  \frac{l_1^2}{4}+
                   \left\{\frac{(l_1+p')^2}{4}\right\} \nonumber \\
                   &   & \phantom{\frac{p_1^2}{4}+\frac{(p_1-l_1)p'}{2}}-
                   (p_1-l_1)\left( \frac{l_1}{2}+\left\{-\frac{l_1}{2}\right\}+
                   1\right) \nonumber \\
                   & = &\!\! \frac{(p_1-l_1)^2}{4}+ (p_1-l_1)\left(\frac{p'}{2}-
                   \left\{ -\frac{l_1}{2} \right\}-1\right)+
                   \left\{ \frac{(l_1+p')^2}{4}\right\}\!\!. \ \ \
\end{eqnarray}
Since $\left\{-\frac{l_1}{2}\right\}+ 1\leq \frac{3}{2}$ it
follows that~\eqref{eq::Thm_H_5} is nonnegative for $p'\geq 3$.

Note that $p'\ne 1$. Otherwise $n=2$ and the only vertex of ${\rm
Inf}\,V_2$ is contained, in view of connectedness of $G$, in ${\rm
Inf}\,V_1$; this is impossible. If $p'=2$,
then~\eqref{eq::Thm_H_5} becomes
\begin{align*}\label{eq::Thm_H_1}
& \frac{(p_1-l_1)^2}{4}-(p_1-l_1)\left\{-\frac{l_1}{2}\right\} +\left\{
\frac{(l_1+2)^2}{4}\right\}=\\
& \frac{(p_1-l_1)^2}{4}-(p_1-l_1)\left\{\frac{l_1}{2}\right\}
+\left\{
\frac{l_1^2}{4}\right\}=\left(\frac{p_1-l_1}{2}-\left\{\frac{l_1}{2}\right\}
\right)^2\geq 0.
\end{align*}

Thus we have completed the proof for $K_4$-free graphs.

Now let $G$ with ${\rm vol}\,G=(p,q)$ be an arbitrary  $H$-free graph, which is not
isomorphic to $K_4$.

Use the induction on the number of subgraphs of $G$ isomorphic to
$K_4$. If there are no such subgraphs, then the statement is
already proved; therefore, the  basis step is verified.

Let $F$ be a subgraph of $G$ isomorphic to $ K_4 $ and $U = \{v_1, v_2, v_3, v_4 \} $ be
the set of its vertices. Let $\overline{F} = G [V \setminus U]$ and ${\rm vol} \,
\overline{F} = (l, m)$. Note that $ {\rm d}(x, U) \leq 1$ for every vertex $x \in V
\setminus U $, otherwise $G$ is not $H$-free.

If $l = 1$, then $p = 5$, $q=7$; therefore $q \leq \frac{p^2}{4}+1$.

If $\overline{F}$ is isomorphic to $K_4$, then ${\rm d}(U,
V\setminus U)\leq 4$, because each of the vertices of $F$ is
adjacent to not more than one vertex of $\overline{F}$. In this
case $p=8$, $q\leq 16<\lfloor\frac{64}{4}\rfloor+1$.

If $\overline{F}$ is not isomorphic to  $K_4$ and $l\geq 2$, then applying the induction
assumption to $\overline{F}$, we have $m\leq \frac{l^2}{4}+1$. Since $p=l+4$ and $q\leq
m+l+6$, it follows that
$$
\frac{p^2}{4}+1-q\geq
\frac{(l+4)^2}{4}+1-m-l-6=\left(\frac{l^2}{4}+1-m\right)+(l-2)\geq
0,
$$
as required.
\end{proof}

\end{document}